\documentclass[11pt]{amsart}

\usepackage[T1]{fontenc}
\usepackage[utf8]{inputenc}
\usepackage{textcomp}
\usepackage{lmodern}
\usepackage{microtype}

\usepackage{amssymb}
\usepackage{amsthm}
\usepackage{amscd}

\usepackage{hyperref}

\usepackage{mathscinet}
\usepackage[giveninits=true,doi=false,url=false,sortcites,backend=biber]{biblatex}
\newtheorem{theorem}{Theorem}[section]

\newtheorem{lemma}[theorem]{Lemma}

\theoremstyle{definition}
\newtheorem{definition}[theorem]{Definition}

\newtheorem{claim}{Claim}

 \newenvironment{claimproof}{\begin{proof}}{\end{proof}}

\allowdisplaybreaks[2]

\begin{document}

\title{Nonstandard Convergence Gives Bounds on Jumps}
\author{Henry Towsner}
\date{\today}
\thanks{Partially supported by NSF grant DMS-1600263}
\address {Department of Mathematics, University of Pennsylvania, 209 South 33rd Street, Philadelphia, PA 19104-6395, USA}
\email{htowsner@math.upenn.edu}
\urladdr{\url{http://www.math.upenn.edu/~htowsner}}

\begin{abstract}
If we know that some kind of sequence always converges, we can ask how quickly and how uniformly it converges.  Many convergent sequences converge non-uniformly and, relatedly, have no computable rate of convergence.  However proof-theoretic ideas often guarantee the existence of a uniform ``meta-stable'' rate of convergence.

We show that obtaining a stronger bound---a uniform bound on the number of jumps the sequence makes---is equivalent to being able to strengthen convergence to occur in the nonstandard numbers.  We use this to obtain bounds on the number of jumps in nonconventional ergodic averages.
\end{abstract}

\maketitle

\section{Introduction}

Once we have proven that some kind of sequence $(a_n)_{n\in\mathbb{N}}$ converges, a natural question is to ask how quickly it converges.  It is not hard to show that there may not be a general rate of convergence\footnote{A \emph{rate of convergence} is a function $F:\mathbb{N}\rightarrow\mathbb{N}$ such that, for each $E>0$ and any $m>F(E)$, $d(a_{F(E)},a_m)<1/E$.}: it might be that in different situations, this sequence converges at substantially different rates, so that no rate of convergence suffices in general.

Indeed, this is the typical situation.  For example, consider the ergodic averages.  We have a probability space $(X,\mathcal{B},\mu)$ and a measurable, measure-preserving $T:X\rightarrow X$.  When $f$ is an $L^1(X)$ function, one can prove that the ergodic averages $A^T_Nf=\frac{1}{N}\sum_{i=0}^{N-1}f(T^ix)$ converge (in the $L^2$ norm \cite{Neumann01011932} and pointwise almost everywhere \cite{1931PNAS...17..656B}).  However it is known that the rate can be arbitrarily slow \cite{MR510630} and non-computable \cite{MR2206254}.

On the other hand, once we have proven convergence, there must be a weaker notion, a rate of metastable convergence\footnote{A \emph{rate of metastable convergence} is a functional $\mathcal{R}:\mathbb{N}\times \mathbb{N}^{\mathbb{N}}\rightarrow\mathbb{N}$ such that for each $E>0$ and each monotone $F:\mathbb{N}\rightarrow\mathbb{N}$, $d(a_{\mathcal{R}(E,F)},a_{F(\mathcal{R}(E,F))})<1/E$.} which is both computable and uniform \cite{avigad:MR2550151,iovino_duenez,MR3141811,tao08Norm,MR2943215,2016arXiv161005397C}.

Avigad and Rute have noted \cite{MR3345161} that we cannot, in general, expect anything more than a rate of metastable convergence.  Kohlenbach and Safarik identified proof-theoretic features of a proof which make it possible to extract a notion intermediate between a rate of convergence and a rate of metastable convergence \cite{MR3111912}: a uniform, computable bound on the number of $\epsilon$-jumps.  In this paper, we give an exact criterion for when this stronger bound can be obtained using nonstandard analysis, and use it to show the existence of such bounds on certain nonconventional ergodic averages.

Our criterion will involve taking a sequence $(a_n)_{n\in\mathbb{N}}$ and extending to a sequence $(a_{\bar n})_{\bar n\in\mathbb{N}^*}$ over the \emph{hypernatural numbers} (that is, the nonstandard natural numbers).  In complete generality, this is not possible: there is no unique choice of hypernatural numbers, and no canonical way to extend a sequence to the hypernatural numbers.

To fix this, we borrow an insight from Avigad and Iovino \cite{MR3141811}.  When we prove convergence, we prove it in some theory with a family of models.  If we formulate the theory in a reasonable way, the models will also be closed under ultraproducts, and so the corresponding sequences will still converge in these ultraproducts.  Furthermore, in any particular ultraproduct, there is a corresponding canonical choice of the hypernatural numbers, and a corresponding canonical extension of a sequence to those hyperreals.

Given the extension $(a_{\bar n})_{\bar n\in\mathbb{N}^*}$, we can ask whether we obtain a stronger kind of convergence: whenever $I\subseteq\mathbb{N}^*$ is a \emph{cut}---an initial segment closed under successor\footnote{The term ``cut'' is sometimes reserved for a stronger notion, adding the requirement that $I$ be closed under addition.  We follow the convention in the models of arithmetic literature, calling this stronger notion an \emph{additive cut}.}---we can ask about \emph{convergence in $I$}: is it the case that, for every (standard real) $\epsilon>0$, there is an $\bar n\in I$ so that, for all $\bar m\in I$ with $\bar m>\bar n$, $d(a_{\bar n},a_{\bar m})<\epsilon$.  (Of course, convergence in $\mathbb{N}$ is just the usual notion of convergence.)

Our main result shows that we obtain uniform bounds on the numbers of $\epsilon$-jumps exactly when these extensions converge in every $I$.  Formally, after giving definitions in Section \ref{sec:main}, we will show:
\begin{theorem}
Let $\mathcal{C}$ be a collection of pairs $((X,d),(a_n)_{n\in\mathbb{N}})$ where each $(a_n)$ is a sequence of elements in the corresponding metric space $(X,d)$.  The following are equivalent:
\begin{itemize}
\item there is a uniform bound on the number of $\epsilon$-jumps; that is, for every $\epsilon>0$ there is a $K$ so that in every pair $((X,d),(a_n)_{n\in\mathbb{N}})$ in $\mathcal{C}$ and every sequence $n_1<n_2<\cdots<n_K$, there is a $k<K$ with $d(a_{n_k},a_{n_{k+1}})<\epsilon$,
\item whenever $\mathcal{U}$ is a nonprincipal ultrafilter on $\mathbb{N}$, and, for each $i$, $((X_i,d_i),(a^i_n)_{n\in\mathbb{N}})\in \mathcal{C}$, in the ultraproduct $\prod_{\mathcal{U}}(X_i,d_i)$, the extended sequence $( a_{\bar n})_{\bar n\in\mathbb{N}^*}$ converges in every cut.
\end{itemize}
\end{theorem}

To illustrate this idea, in Section \ref{sec:ergodic_in_cuts} we will show that, with a small modification, the original proof of the mean ergodic theorem satisfies the criterion given by the second equivalent condition in this theorem.  (The existence of a bound on jumps for this sequence already follows from Bishop's upcrossing inequalities \cite{bishop:MR0228655}.)

We will then turn to the ``nonconventional'' ergodic averages
\[\frac{1}{N}\sum_{n=1}^N (f_1\circ T_1^n)\cdots(f_k\circ T_k^n).\]
These averages were shown to converge by Tao \cite{tao08Norm}.  Although there are now several proofs of convergence \cite{MR2539560}, including a proof using nonstandard analysis \cite{towsner:MR2529651}, we will modify Austin's proof \cite{MR2599882} to show:
\begin{theorem}
  For every $d$ and every $\epsilon>0$, there is a $K$ so that whenever $(X,\mathcal{B},\mu)$ is a probability measure space, $T_1,\ldots,T_d:X\rightarrow X$ are a sequence of measurable, measure-preserving transformations, and $f_1,\ldots,f_d$ are functions with each $||f_i||_{L^\infty}\leq 1$, whenever $N_1<N_2<\cdots<N_K$ are given, there is an $i<K$ so that
\[||\frac{1}{N_i}\sum_{n=1}^{N_i} (f_1\circ T_1^n)\cdots(f_k\circ T_k^n)-\frac{1}{N_{i+1}}\sum_{n=1}^{N_{i+1}} (f_1\circ T_1^n)\cdots(f_k\circ T_k^n)||_{L^2}<\epsilon.\]
\end{theorem}
The $d=1$ case is the regular ergodic theorem, and for the $d=2$ case stronger variational inequalities are known \cite{MR2417419,MR3224563}.

Finally, in Section \ref{sec:distances} we consider a generalization where we only restrict those jumps where $n_k$ and $n_{k+1}$ are ``far apart'', and show that this corresponds to a weaker condition where the extended sequences only converge in cuts with additional closure properties.

\section{Nonstandard Convergence}\label{sec:main}

\subsection{Ultraproducts}

Throughout, we assume that $\mathcal{U}$ is a nonprincipal ultrafilter on $\mathbb{N}$; essentially nothing would change if we replaced $\mathcal{U}$ with a nonprincipal, non-countably complete ultrafilter on some larger index set.
 
We recall the notion of an ultraproduct for metric structures.  A detailed exposition is given in \cite{MR2436146}.

\begin{definition}
  For each $i\in\mathbb{N}$, let $r_i\in\mathbb{R}$.  If there is some $B$ so that $\{i\mid |r_i|\leq B\}\in \mathcal{U}$ then the \emph{ultralimit} $\lim_{\mathcal{U}}r_i$ is defined to be the unique $r\in\mathbb{R}$ such that, for every $\epsilon>0$, $\{i\mid |r_i-r|<\epsilon\}\in\mathcal{U}$.  

For each $i\in \mathbb{N}$, let $(X_i,d_i)$ be a metric space.  The \emph{metric ultraproduct} $\prod_{\mathcal{U}}(X_i,d_i)$ is a metric space $(X_{\mathcal{U}},d_{\mathcal{U}})$ given by:
\begin{itemize}
\item $Y_{\mathcal{U}}$ consists of sequences $\langle x_i\rangle_{i\in \mathbb{N}}$ where $x_i\in X_i$ for each $i\in \mathbb{N}$,
\item we define an equivalence relation $\sim_{\mathcal{U}}$ on $Y_{\mathcal{U}}$ by $\langle x_i\rangle_{i\in \mathbb{N}}\sim_{\mathcal{U}}\langle y_i\rangle$ iff $\lim_{\mathcal{U}}d_i(x_i,y_i)=0$,
\item $X_{\mathcal{U}}=Y_{\mathcal{U}}/\sim_{\mathcal{U}}$,
\item $d_{\mathcal{U}}([\langle x_i\rangle_{i\in \mathbb{N}}],[\langle y_i\rangle_{i\in \mathbb{N}}])=\lim_{\mathcal{U}}d_i(x_i,y_i)$ if this exists and $\infty$ otherwise.
\end{itemize}
\end{definition}

\begin{definition}
  When $\mathcal{U}$ is an ultrafilter, a \emph{nonstandard natural number} (relative to $\mathcal{U}$) is an equivalence class of sequences $\langle n_i\rangle_{i\in \mathbb{N}}$ where each $n_i\in\mathbb{N}$, taking $\langle n_i\rangle\sim_{\mathcal{U}}\langle m_i\rangle$ iff $\{i\mid n_i=m_i\}\in\mathcal{U}$.
\end{definition}

When $\mathcal{U}$ is clear from context, we write $\mathbb{N}^*$ for the set of nonstandard natural numbers.  The nonstandard integers are defined similarly, and we sometimes write $\mathbb{Z}^*$ for the nonstandard integers.  Recall that $\mathbb{N}$ embeds canonically as an initial segment of $\mathbb{N}^*$ by associating any $n\in\mathbb{N}$ with the constant sequence $[\langle n\rangle_{i\in \mathbb{N}}]\in\mathbb{N}^*$.  Of course, $\mathbb{N}$ is a proper initial segment; for instance, $[\langle i\rangle_{i\in\mathbb{N}}]$ is larger than any element of (the image of) $\mathbb{N}$.

\begin{definition}
  Suppose that, for each $i\in \mathbb{N}$, $\langle a^i_n\rangle_{n\in\mathbb{N}}$ is a sequence of elements of $X_i$.  Then, for any nonstandard natural number $\bar n=[\langle n_i\rangle]$, we define $a_{\bar n}=[\langle a^i_{n_i}\rangle]$.
\end{definition}

It is easy to see that $a_{\bar n}$ is well-defined: if $\langle n_i\rangle$ and $\langle m_i\rangle$ represent the same element of $\mathbb{N}^*$, so $\{i\mid n_i=m_i\}\in\mathcal{U}$ then $\{i\mid d_i(a^i_{n_i},a^i_{m_i})=0\}\in\mathcal{U}$, and therefore $[\langle a^i_{n_i}\rangle]=[\langle a^i_{m_i}\rangle]$.

\begin{definition}
  A \emph{cut} in $\mathbb{N}^*$ is a subset $I\subseteq\mathbb{N}^*$ which is an initial segment, and such that whenever $\bar n\in I$, also $\bar n+1\in I$.
\end{definition}

$\mathbb{N}$ and $\mathbb{N}^*$ are the smallest and largest cuts, respectively.  For more interesting examples, whenever $\bar n\in\mathbb{N}$,
\begin{itemize}
\item $\{\bar m\mid \exists k\in\mathbb{N}\ \bar m<\bar n+k\}$,
\item $\{\bar m\mid \exists k\in\mathbb{N}\ \bar m<k\cdot \bar n\}$, and
\item $\{\bar m\mid \exists k\in\mathbb{N}\ \bar m<\bar n^k\}$
\end{itemize}
are also cuts.

\begin{definition}
 When $I$ is a cut, we say a sequence $(a_{\bar n})_{\bar n\in\mathbb{N}^*}$ \emph{converges in $I$} if for every real $\epsilon>0$, there is an $\bar n\in I$ so that, for all $\bar m\in I$ with $\bar m>\bar n$, $d(a_{\bar n},a_{\bar m})<\epsilon$.
\end{definition}

\subsection{Main Theorem}

\begin{definition}
  Let $(a_n)_{n\in\mathbb{N}}$ be a sequence of elements in some metric space.  For any $\epsilon>0$, we say $(a_n)$ \emph{admits $K$ $\epsilon$-jumps} if there are $n_1<n_2<\cdots<n_K$ such that, for each $k<K$, $d(a_{n_k},a_{n_{k+1}})\geq\epsilon$. 
\end{definition}

\begin{theorem}\label{thm:main}
Let $\mathcal{C}$ be a collection of pairs $((X,d),(a_n)_{n\in\mathbb{N}})$ where each $(a_n)$ is a sequence of elements in the corresponding metric space $(X,d)$.  The following are equivalent:
\begin{itemize}
\item for every $\epsilon>0$ there is a $K$ so that, for every $((X,d),(a_n)_{n\in\mathbb{N}})\in \mathcal{C}$, the sequence $(a_n)_{n\in\mathbb{N}}$ does not admit $K$ $\epsilon$-jumps,
\item whenever $\mathcal{U}$ is a nonprincipal ultrafilter on $\mathbb{N}$, and, for each $i$, $((X_i,d_i),(a^i_n)_{n\in\mathbb{N}})\in \mathcal{C}$, the sequence $(a_{\bar n})_{\bar n\in\mathbb{N}^*}$ converges in every cut in $(X_{\mathcal{U}},d_{\mathcal{U}})$.
\end{itemize}
\end{theorem}
\begin{proof}
  Suppose the former fails: there is some $\epsilon>0$ so that, for every $K$, there is an $((X,d),(a_n)_{n\in\mathbb{N}})\in \mathcal{C}$ so that $(a_n)_{n\in\mathbb{N}}$ admits $K$ $\epsilon$-jumps.  For each $K$, choose such an $((X^K,d^K),(a^K_n)_{n\in\mathbb{N}})$ and choose witnesses $n^K_1<n^K_2<\cdots<n^K_K$ so that, for each $k<K$, $d(a^K_{n^K_k},a^K_{n^K_{k+1}})\geq\epsilon$.

Take any nonprincipal ultrafilter $\mathcal{U}$ and consider the sequence $(a_{\bar n})_{\bar n\in\mathbb{N}^*}$.  For each $i\in\mathbb{N}$, take $\bar n_i=[\langle n^K_i\rangle_{K\in\mathbb{N}}]$ (where we take $n^K_i=0$ if $K<i$; for each $i$, there are only finitely many such $K$, so this arbitrary choice does not affect the value of $\bar n_i$).  Let $I=\{\bar m\mid \exists i\ \bar m<\bar n_i\}$; this is a cut, since if $\bar m<\bar n_i$ then $\bar m+1<\bar n_i+1\leq \bar n_{i+1}$.  But for each $i$, $d_{\mathcal{U}}(a_{\bar n_i},a_{\bar n_{i+1}})=\lim_{\mathcal{U}}d_K(a^K_{n^K_i},a^K_{n^K_{i+1}})\geq\epsilon$.  Therefore $(a_{\bar n})_{\bar n\in\mathbb{N}^*}$ does not converge to within $\epsilon/2$ in the cut $I$: given any $\bar m\in I$, we may find some $\bar n_i>\bar m$ by definition, and by the triangle inequality, either $d_{\mathcal{U}}(a_{\bar m},\bar n_i)\geq\epsilon/2$ or $d_{\mathcal{U}}(a_{\bar m},\bar n_{i+1})\geq\epsilon/2$.

Conversely, suppose the former holds, and consider any nonprincipal ultrafilter $\mathcal{U}$, any sequence $((X_i,d_i),(a^i_n)_{n\in\mathbb{N}})$ of elements of $\mathcal{C}$.  Consider some cut $I$ and some $\epsilon>0$.  Let $K$ witness the uniform bound for $\epsilon/2$-jumps.  Choose any $\bar n_1\in I$; if $\bar n_1$ does not witness convergence in $I$, there must be some $\bar n_2>\bar n_1$ with $\bar n_2\in I$ and $d_{\mathcal{U}}(\bar n_1,\bar n_2)\geq\epsilon$.  We continue choosing $\bar n_3>\bar n_2$ in $I$, and so on.  If we find $k<K$ so that $\bar n_k$ witnesses convergence in $I$, we are done.  Otherwise, we find $\bar n_1<\bar n_2<\cdots<\bar n_K$ all in $I$ with $d_{\mathcal{U}}(a_{\bar n_k},a_{\bar n_{k+1}})\geq\epsilon$ for each $k<K$.

Then, for each $k<K$, $\lim_{\mathcal{U}}d_i(a^i_{n^i_k},a^i_{n^i_{k+1}})\geq\epsilon$.  In particular, $\{i\mid d_i(a^i_{n^i_k},a^i_{n^i_{k+1}})>\epsilon/2\}\in\mathcal{U}$.  So we may choose a single $i$ so that, for all $k<K$ simultaneously, $d_i(a^i_{n^i_k},a^i_{n^i_{k+1}})>\epsilon/2$ and $n^i_1<n^i_2<\cdots<n^i_K$.  But this shows that $(a^i_n)_{n\in\mathbb{N}}$ admits $K$ $\epsilon/2$-jumps, which is a contradiction.  So it must be that, for some $k<K$, $\bar n_k$ witnessed convergence in $I$.
\end{proof}


\section{The Mean Ergodic Theorem}\label{sec:ergodic_in_cuts}

As a warm up (and to establish our base case), we modify the proof of the mean ergodic theorem to hold in every cut in an ultraproduct.  The proof does \textit{not} go through unchanged.  Like most proofs, von Neumann's proof of the mean ergodic theorem requires a certain amount of arithmetic, which amounts to saying that it only goes through in ``nice enough'' cuts.  In this case the condition is mild: von Neumann's argument needs the cut to be \emph{additive} (that is, closed under addition).  However, as we will show, averages always converge in non-additive cuts, so we are able to complete the proof.  (This dichotomy between additive and non-additive cuts should be compared to the ``gap condition'' appearing in \cite{MR3111912}.)

We first note that averages always converge in non-additive cuts.
\begin{definition}
  A cut $I\subseteq\mathbb{N}^*$ is \emph{additive} if whenever $\bar n,\bar m\in I$, also $\bar n+\bar m\in I$.
\end{definition}

\begin{lemma}\label{thm:non_additive}
  Let $I$ be a non-additive cut, let $(c_{\bar n})_{\bar n\in I}$ be a sequence of elements of $L^2(X)$ with norm bounded by $1$, and let $a_{\bar N}=\frac{1}{\bar N}\sum_{\bar n=1}^{\bar N}c_{\bar n}$.  Then the sequence $(a_{\bar N})$ converges in $I$.
\end{lemma}
\begin{proof}
  For any $\epsilon>0$, we choose $\bar N\in I$ so that $\lfloor\frac{\bar N}{1-\epsilon/2}\rfloor\not\in I$; such a $\bar N$ exists because $I$ is not additive.  Then whenever $\bar M>\bar N$ belongs to $I$, we have $\bar M<\frac{\bar N}{1-\epsilon/2}$ and therefore
  \begin{align*}
    ||a_{\bar N}-a_{\bar M}||_{L^2}
&=||\frac{1}{\bar N}\sum_{\bar n=1}^{\bar N}c_{\bar n}-\frac{1}{\bar M}\sum_{\bar m=1}^{\bar M}c_{\bar m}||_{L^2}\\
&=\frac{\bar M-\bar N}{\bar M\bar N}||\sum_{\bar n=1}^{\bar N}c_{\bar n}||_{L^2}+\frac{1}{\bar M}||\sum_{\bar m=\bar N+1}^{\bar M}c_{\bar m}||_{L^2}\\
&=2\frac{\bar M-\bar N}{\bar M}\\
&<\epsilon.
  \end{align*}
\end{proof}

\begin{definition}
Let $(X,\mu)$ be a probability measure space and $T:\mathbb{Z}\curvearrowright (X,\mu)$ a measurable, measure-preserving action.  For any $f\in L^1(X)$, we define the \emph{ergodic average} $A^T_Nf$ by $(A^T_Nf)(x)=\frac{1}{N}\sum_{n=1}^Nf(T^nx)$.
\end{definition}

\begin{theorem}\label{thm:mean}
For every $\epsilon>0$ there is a $K$ so that whenever $(X,\mu)$ be a probability measure space and $T:\mathbb{Z}\curvearrowright(X,\mu)$ a measurable, measure-preserving action and $f\in L^2(\mu)$ with $||f||_{L^2}\leq 1$, the sequence $A^T_Nf$ does not admit $K$ $\epsilon$-jumps in the $L^2$ norm.
\end{theorem}
\begin{proof}
We take $\mathcal{C}$ to be the collection of all pairs of the form $((L^2(\mu),d_{L^2}),(A^T_Nf)_{N\in\mathbb{N}})$ where $d_{L^2}(f,g)=||f-g||_{L^2(\mu)}$ and $f\in L^2(\mu)$ with $||f||_{L^2}\leq 1$.  Consider an ultraproduct of elements from $\mathcal{C}$; it is the $L^2$ space of a probability measure space $(X,\mu)$ with a measurable, measure-preserving $T:\mathbb{Z}^*\curvearrowright(X,\mu)$.  Given a function $f\in L^2(\mu)$ with $||f||_{L^2}\leq 1$, we can consider the sequence $(A^T_{\bar N}f)_{\bar N\in\mathbb{N}^*}$.

Consider a cut $I$.  If $I$ is not additive, convergence follows from the previous lemma, so assume $I$ is additive.  Then whenever $C\in\mathbb{N}$ and $\bar n\in I$, $C\bar n\in I$.

Consider the space $\mathcal{N}\subseteq L^2(\mu)$ spanned by functions of the form $f-f\circ T^{\bar n}$ for $\bar n\in I$.  Whenever $g\in\mathcal{N}$, for any $\epsilon>0$ we may write $g=\sum_{i=0}^kc_i(f-f\circ T^{\bar n_i})+g^-$ where $||g^-||_{L^2}<\epsilon/2$, $k\in\mathbb{N}$, and each $c_i\in\mathbb{R}$.  Whenever $\bar M>\frac{4\sum_{i=0}^k|c_i|\bar n_i}{\epsilon}$ we have
\begin{align*}
||A^T_{\bar M}g||_{L^2}
&=||\frac{1}{\bar M}\sum_{\bar m=1}^{\bar M}\sum_{i=0}^kc_i(f\circ T^{\bar m}-f\circ T^{\bar n_i+\bar m})+A^T_{\bar M}g^-||_{L^2}\\
&<\sum_{i=0}^k|c_i|\cdot||\frac{1}{\bar M}\sum_{\bar m=1}^{\bar M}(f\circ T^{\bar m}-f\circ T^{\bar n_i+\bar m})||_{L^2}+\epsilon/2\\
&\leq\sum_{i=0}^k|c_i|\frac{2\bar n_i}{\bar M}||f||_{L^2}+\epsilon/2\\
&\leq\epsilon.
\end{align*}
In particular, since $\lceil\frac{4\sum_{i=0}^k|c_i|\bar n_i}{\epsilon}\rceil\in I$, for $g\in \mathcal{N}$, $A^T_{\bar N}g$ converges to $0$ in $I$.

There is a projection $f_0=\mathbb{E}(f\mid \mathcal{N})$.  Let $f^-=f-f_0$.  For any $\bar n\in I$, observe that, using the invariance of $T^{\bar n}$,
\begin{align*}
  ||f^--f^-\circ T^{\bar n}||_{L^2}^2
&=\langle f^-,f^-\rangle-2\langle f^-,f^-\circ T^{\bar n}\rangle+\langle f^-\circ T^{\bar n},f^-\circ T^{\bar n}\rangle\\
&=2(\langle f^-,f^-\rangle-\langle f^-,f^-\circ T^{\bar n}\rangle)\\
&=2\langle f^-,f^--f^-\circ T^{\bar n}\rangle\\
&=0
\end{align*}
because $f^--f^-\circ T^{\bar n}\in\mathcal{N}$.  Therefore
\[||A^T_{\bar M}f-f^-||_{L^2}=||A^T_{\bar M}f_0||_{L^2}+||f^--f^-||_{L^2}\]
approaches $0$ as $\bar M$ gets large in $I$.  In particular, $A^T_{\bar M}f$ converges to $f^-$ in the cut $I$.

By Theorem \ref{thm:main}, we obtain uniform bounds on $\epsilon$-jumps.
\end{proof}

\section{Nonconventional Ergodic Averages}

\subsection{Preliminaries}

\begin{definition}
 When $I$ is a cut in $\mathbb{N}^*$, we write $\mathbb{Z}(I)$ for $\{z\in\mathbb{Z}^*\mid |z|\in I\}$.

  Let $(X,\mu)$ be a probability measure space and $T:\mathbb{Z}(I)^d\curvearrowright(X,\mu)$ a measurable, measure-preserving action.  We write $T_i^{\bar n}:X\rightarrow X$ for the action $T^{(0,\ldots,\bar n,\ldots,0)}$ with the $\bar n$ in the $i$-th position and abbreviate $f\circ T^{\bar n}_i$ by $T^{-\bar n}_if$.

We define
\[A_{\bar N}^{T}(f_1,\ldots,f_d)=\frac{1}{\bar N}\sum_{\bar n=1}^{\bar N}\prod_{1\leq i\leq d} T^{-\bar n}_if_i.\]
\end{definition}

Of course, we are primarily interested in the case where $I=\mathbb{N}$.  The average appearing in the ergodic theorem is then the case where $d=1$.

\begin{definition}
  When a sequence of functions $f_{\bar N}$ converges in $I$ in the $L^2$ norm, we write $\lim_{\bar N\rightarrow I}f_{\bar N}$ for the $L^2$-limit of these functions.
\end{definition}

We observe that the van der Corput trick holds in any additive cut, following the standard proof without change.
\begin{lemma}[van der Corput]
  Suppose that $a_{\bar n}\in L^2(X)$ with $L^2$ norm bounded by $1$ for all $\bar n\in I$ where $I$ is an additive cut.
  If
\[\lim_{H\rightarrow I}\limsup_{\bar N\rightarrow I}\frac{1}{\bar H}\sum_{\bar h=1}^{\bar H}\left|\frac{1}{\bar N}\sum_{\bar n=1}^{\bar N}\int a_{\bar n+\bar h}a_{\bar n}\,d\mu\right|=0\]
then
\[\lim_{\bar N\rightarrow I}\left|\frac{1}{\bar N}\sum_{\bar n=1}^{\bar N}a_{\bar n}\right|=0.\]
\end{lemma}

Following Austin's proof \cite{MR2599882} that the averages $A^T_N(f_1,\ldots,f_d)$ converge, we proceed by induction on $d$.  The case $d=1$ is the mean ergodic theorem discussed above.


  

\begin{lemma}
  Let $I$ be an additive cut and let $T:\mathbb{Z}(I)^d\curvearrowright(X,\mu)$ be measure-preserving.

Whenever $f_1\in L^2(X)$ and $f_2,\ldots,f_d\in L^\infty(X)$, the averages $A^T_{\bar N}(f_1,\ldots,f_d)$ converge in $I$.
\end{lemma}
\begin{proof}
We proceed by induction on $d$.  The $d=1$ case is shown in Theorem \ref{thm:mean}, so we assume $d>1$ and the claim holds for $d-1$.  Let $f_1,\ldots,f_d$ be given.

We will repeatedly need the averages
\[\hat A_{\bar N}(g)=A^T_{\bar N}(g,f_2,\ldots,f_d).\]

\begin{claim}
For every $g$, the function
\[u_g=\lim_{\bar H\rightarrow I}\frac{1}{\bar H}\sum_{\bar h=1}^{\bar H}\lim_{\bar S\rightarrow I}\frac{1}{\bar S}\sum_{\bar s=1}^{\bar S}T_1^{-\bar h}g\prod_{1<i\leq d}(T_1^{-1}T_i)^{-\bar s}(f_iT_i^{-\bar h}f_i)\]
exists.
\end{claim}
\begin{claimproof}
We define functions
\[u_{g,\bar H}=\frac{1}{\bar H}\sum_{\bar h=1}^{\bar H}T_1^{-\bar h}g\lim_{\bar S\rightarrow I}\frac{1}{\bar S}\sum_{\bar s=1}^{\bar S}\prod_{1<i\leq d}(T_1^{-1}T_i)^{-\bar s}(f_iT_i^{-\bar h}f_i).\]
For each $\bar h$, we can show that the limit exists using the inductive hypothesis applied to the transformation $U:\mathbb{Z}(I)^{d-1}\curvearrowright(X,\mu)$ given by $U_i=T_1^{-1}T_{i+1}$.

We now use the inductive hypothesis to show that this sequence of functions also converges for each $g\in L^2(X)$.  We use a modified version of the \emph{Furstenberg self-joining}.  We define a measure $\mu^{\oplus d,I}$ on $X^d$ by setting
\[\mu^{\oplus d,I}(\prod_{1\leq i\leq d}B_i)=\int \chi_{B_1}\lim_{\bar S\rightarrow I}\frac{1}{\bar S}\sum_{\bar s=1}^{\bar S}\prod_{1<i\leq d}(T_1^{-1}T_i)^{-{\bar s}}\chi_{B_i}\,d\mu.\]
The inductive hypotheis guarantees that this limit exists.  We define $\tilde T:\mathbb{Z}(I)^d\curvearrowright (X^d,\mu^{\oplus d,I})$ by $\tilde T_1(x_1,\ldots,x_d)=(T_1x_1,T_2x_2,\ldots,T_dx_d)$ and, for $1<i\leq d$, $\tilde T_i(x_1,\ldots,x_d)=(T_ix_1,\ldots,T_ix_d)$.  Since $T$ is measure-preserving, $\tilde T$ is as well; for $\tilde T_i$ with $1<i$, this is immediate.  To see that $\tilde T_1$ is measure-preserving, observe that
\begin{align*}
  \mu^{\oplus d,I}(T_1^{-\bar n}\prod_{1\leq i\leq d}B_i)
&=\int \chi_{B_1}(T_1^{\bar n}x) \lim_{\bar S\rightarrow I}\frac{1}{\bar S}\sum_{\bar s=1}^{\bar S}\prod_{1<i\leq d}(T_1^{-1}T_i)^{-\bar s}\chi_{B_i}(T_i^{\bar n}x)\,d\mu\\
&=\int \chi_{B_1}(T_1^{\bar n}x) \lim_{\bar S\rightarrow I}\frac{1}{\bar S}\sum_{\bar s=1}^{\bar S} \prod_{1<i\leq d}(T_1^{-1}T_i)^{-\bar s-\bar n}\chi_{B_i}(T_1^{\bar n}x)\,d\mu\\
&=\int \chi_{B_1}(x) \lim_{\bar S\rightarrow I}\frac{1}{\bar S}\sum_{\bar s=1}^{\bar S}\prod_{1<i\leq d}(T_1^{-1}T_i)^{-\bar s-\bar n}\chi_{B_i}(x)\,d\mu\\
&=\int \chi_{B_1}(x) \lim_{\bar S\rightarrow I}\frac{1}{\bar S}\sum_{\bar s=1}^{\bar S}\prod_{1<i\leq d}(T_1^{-1}T_i)^{-\bar s}\chi_{B_i}(x)\,d\mu.\\
\end{align*}
Note that the last step uses the fact that $I$ is additive.

Consider the function $\tilde g(x_1,\ldots,x_d)=g(x_1)\prod_{1< i\leq d}f_i(x_i)$.  
By Theorem \ref{thm:mean}, the averages $A_{\bar N}^{\tilde T_1}(\tilde g)$ converge in $I$.  Observe that this also shows that the sequence $A_{\bar N}^{\tilde T_1}(\tilde g)\tilde f$ converges in $I$.  Consider the projection onto the first coordinate (defined up to $L^2$ norm) given by $\mathbb{P}(\prod_{1\leq i\leq d}B_i)(x)=\chi_{B_1}(x) \lim_{\bar S\rightarrow I}\frac{1}{\bar S}\sum_{\bar s=1}^{\bar S}\prod_{1<i\leq d}(T_1^{-1}T_i)^{-\bar s}\chi_{B_i}(x)$.
\begin{align*}
\mathbb{P}(A_{\bar H}^{\tilde T_1}(\tilde g)\tilde f)(x)
&=\mathbb{P}(\frac{1}{\bar H}\sum_{\bar h=1}^{\bar H}g(T_1^{\bar h}x_1)\prod_{1<i\leq d}f_i(T_i^{\bar h}x_i)f_i(x_i))\\
&=\frac{1}{\bar H}\sum_{\bar h=1}^{\bar H}g(T_1^{\bar h}x) \lim_{\bar S\rightarrow I}\frac{1}{\bar S}\sum_{\bar s=1}^{\bar S}\prod_{1<i\leq d}f_i(T_i^{\bar h+\bar s}x)f_i(T_i^{\bar s}x)\\
&=u_{g,\bar H}.
\end{align*}
Since $\mathbb{P}$ is a contraction from $L^2(X^d)$ to $L^2(X)$ and the sequence $A_{\bar H}^{\tilde T_1}(\tilde g)\tilde f$ converges, the sequence $u_{g,\bar H}$ also converges.

\end{claimproof}

We let $\mathcal{N}$ be the linear subspace of $L^2(X)$ generated by functions of the form $u_g$.  

\begin{claim}
For $g\in\mathcal{N}$, $\hat A_{\bar N}(g)$ converges in $I$.
\end{claim}
\begin{claimproof}
It suffices to show that every $\hat A_{\bar N}(u_g)$ converges.  To see this, we again compare to a corresponding sequence in the self-joining.

First, observe that
\begin{align*}
\hat A_{\bar N}(u_g)
&=\frac{1}{\bar N}\sum_{\bar n=1}^{\bar N}\lim_{\bar H\rightarrow I}\frac{1}{\bar H}\sum_{\bar h=1}^{\bar H}\lim_{\bar S\rightarrow I}\frac{1}{\bar S}\sum_{\bar s=1}^{\bar S}T_1^{-\bar h-\bar n}g\prod_{1<i\leq d}T_1^{-\bar n}(T_1^{-1}T_i)^{-\bar s}(f_iT_i^{-\bar h}f_i)\prod_{1<i\leq d}T_i^{-\bar n}f_i\\
&=\frac{1}{\bar N}\sum_{\bar n=1}^{\bar N}\lim_{\bar H\rightarrow I}\frac{1}{\bar H}\sum_{\bar h=1}^{\bar H}\lim_{\bar S\rightarrow I}\frac{1}{\bar S}\sum_{\bar s=1}^{\bar S}T_1^{-\bar h-\bar n}g\prod_{1<i\leq d}(T_1^{-1}T_i)^{-\bar s}(T_i^{-\bar n}f_iT_i^{-\bar h-\bar n}f_i) \prod_{1<i\leq d}T_i^{-\bar n}f_i\\
\end{align*}

Define $\tilde f_1(x_1,\ldots,x_d)=\lim_{\bar H\rightarrow I}\frac{1}{\bar H}\sum_{\bar h=1}^{\bar H}\tilde T_1^{-\bar h}\tilde g$ and, for $1<i\leq d$ let $\tilde f_i(x_1,\ldots,x_d)=f_i(x_1)f_i(x_d)$.  Observe that, since $\tilde f_1$ is $\tilde T_1$-invariant, $A_{\bar N}^{\tilde T}(\tilde f_1,\ldots,\tilde f_d)=\tilde f_1 A_{\bar N}^{\tilde T}(1,f_2,\ldots,f_d)$ which converges by the inductive hypothesis.  

By choosing $\bar H$ large enough, $||A_{\bar H}^{\tilde T_1}(\tilde g)-\tilde f_1||_{L^2(X^d)}$ is small, and therefore 
\begin{align*}
&\lim_{\bar H\rightarrow I}||A_{\bar N}(A_{\bar H}^{\tilde T_1}(\tilde g),\tilde f_2,\ldots,\tilde f_d)-A_{\bar N'}(A_{\bar H}^{\tilde T_1}(\tilde g),\tilde f_2,\ldots,\tilde f_d)||_{L^2(X^d)}\\
&=||A_{\bar N}(\tilde f_1,\ldots,\tilde f_d)-A_{\bar N'}(\tilde f_1,\ldots,\tilde f_d)||_{L^2(X^d)}\end{align*}
uniformly in $\bar N$.

Since
\begin{align*}
  &\lim_{\bar H\rightarrow I}\mathbb{P}(A_{\bar N}(A_{\bar H}^{\tilde T_1}(\tilde g),\tilde f_2,\ldots,\tilde f_d))\\
&=\lim_{\bar H\rightarrow I}\mathbb{P}(\frac{1}{\bar N}\sum_{\bar n=1}^{\bar N}\frac{1}{\bar H}\sum_{\bar h=1}^{\bar H}g(T_1^{\bar h+\bar n}x_1)\prod_{1<i\leq d}f_i(T_i^{\bar h+\bar n}x_i)f_i(T_i^{\bar n}x_1)f_i(T_i^{\bar n}x_i))\\
&=\frac{1}{\bar N}\sum_{\bar n=1}^{\bar N}\lim_{\bar H\rightarrow I}\frac{1}{\bar H}\sum_{\bar h=1}^{\bar H}\lim_{\bar S\rightarrow I}\frac{1}{\bar S}\sum_{\bar s=1}^{\bar S}T_1^{-\bar h-\bar n}g T^{-\bar n}_if_i\prod_{1<i\leq d}(T_1^{-1}T_i)^{-\bar s}(T_i^{-\bar n}f_iT^{-\bar h-\bar n}f_i)\\
&=\hat A_{\bar N}(u_g),
\end{align*}
also $\hat A_{\bar N}(u_g)$ converges.

\end{claimproof}

Now consider the function $f_1$.  We write $f^-=f_1-\mathbb{E}(f_1\mid\mathcal{N})$.  Suppose $\hat A_{\bar N}(f^-)$ does not converge to $0$; then by van der Corput, there is an $\epsilon>0$ so that we may find sufficiently large $\bar H$ so that
\[\epsilon^2<\limsup_{\bar S\rightarrow I}\left|\int \frac{1}{\bar H}\sum_{\bar h=1}^{\bar H}\frac{1}{\bar S}\sum_{\bar s=1}^{\bar S} T_1^{-\bar h-\bar s}f^-T_1^{-\bar s}f^-\prod_{1< i\leq d}T_i^{-\bar h-\bar s}f_iT_i^{-\bar s}f_id\mu\right|.\]
Shifting each term by $T_1^{\bar s}$,
\[\epsilon^2<\limsup_{\bar S\rightarrow I}\left|\int f^-\frac{1}{\bar H}\sum_{\bar h=1}^{\bar H}\frac{1}{\bar S}\sum_{\bar s=1}^{\bar S} T_1^{-\bar h}f^-\prod_{1< i\leq d}(T_1^{-1}T_i)^{-\bar s}(f_iT_i^{-\bar h}f_i)d\mu\right|.\]
Since this holds for sufficiently large $\bar H$, $|\int f^- u_{f^-}d\mu|>0$.  But this contradicts the fact that $\mathbb{E}(f^-\mid\mathcal{N})=0$.

So
\[\hat A_{\bar N}(f_1)=\hat A_{\bar N}(f^-)+\hat A_{\bar N}(\mathbb{E}(f_1\mid\mathcal{N}))\rightarrow\lim_{\bar N\rightarrow I}\mathbb{E}(f_1\mid\mathcal{N}).\]
\end{proof}

Combining this with Lemma \ref{thm:non_additive} and Theorem \ref{thm:main}, we obtain:
\begin{theorem}
  For every $\epsilon>0$ there is a $K$ so that whenever $T:\mathbb{Z}^d\curvearrowright(X,\mu)$ is measure-preserving, $||f_1||_{L^2}\leq 1$, and $||f_i||_{L^\infty}\leq 1$ for $1<i\leq d$, the sequence $A^T_N(f_1,\ldots,f_d)$ does not admit $K$ $\epsilon$-jumps.
\end{theorem}

\section{Bounded Jumps over Long Distances}\label{sec:distances}

One might think that, in the previous two sections, we were fortunate that the nature of averages gives convergence in non-additive cuts: the true underlying arguments by von Neumann and Austin only pertained to additive cuts, and it was an incidental feature that we could handle non-additive cuts by another means.

A natural generalization of convergence in all cuts is to consider convergence only in cuts with suitable closure properties---say, only additive cuts, or only cuts closed under exponentiation.  This corresponds to a variant of bounding the number of jumps where we only consider jumps between elements which are sufficiently far apart.

\begin{definition}
Let $(a_n)_{n\in\mathbb{N}}$ be a sequence of elements in some metric space and let $h:\mathbb{N}\rightarrow\mathbb{N}$ be a weakly increasing function with $n< h(n)$ for all $n$.  For any $\epsilon>0$, we say $(a_n)$ \emph{admits $K$ $\epsilon$-jumps of distance $h$} if there are $n_1<n_2<\cdots<n_K$ such that, for each $k<K$, $h(n_k)\leq n_{j+1}$ and $d(a_{n_k},a_{n_{k+1}})\geq\epsilon$.
\end{definition}
Although phrased differently, failing to admit $K$ $\epsilon$-jumps of distance $h$ is essentially Kohlenbach and Safarik's notion of effevtive learnability \cite{MR3111912}.  Admitting $K$ $\epsilon$-jumps is the same as taking $h(n)=n+1$.

Considering only jumps of distance $h$ allows for the situation where a sequence can have brief windows with many oscillations, but other than these windows has only boundedly many jumps.  In practice, one usually thinks of $h$ as varying with $\epsilon$---say, $h(n)=2^{r(\epsilon)}n$ for all $n$---but this can be absorbed into $K$, since admitting $K$ $\epsilon$-jumps of distance $h$ is equivalent to admitting roughly $\lfloor K/2r\rfloor$ $\epsilon/2$-jumps of distance $h^r$ (the iteration of $h$ $r$-times).

It is known that some cases exist with bounds of this more general kind which cannot be improved to a bound on jumps \cite{MR3438729}.

\begin{definition}
  If $h:\mathbb{N}\rightarrow\mathbb{N}$, we write $\bar h:\mathbb{N}^*\rightarrow\mathbb{N}^*$ for the function given by $\bar h([\langle n_i\rangle])=\langle h(n_i)\rangle$.  

We say a cut $I\subseteq\mathbb{N}^*$ is \emph{ closed under $h$} if for every $\bar n\in I$, $\bar h(\bar n)\in I$.
\end{definition}

\begin{theorem}\label{thm:main_dist}
Let $\mathcal{C}$ be a collection of pairs $((X,d),(a_n)_{n\in\mathbb{N}})$ where each $(a_n)$ is a sequence of elements in the corresponding metric space $(X,d)$.  For any weakly increasing function $h:\mathbb{N}\rightarrow\mathbb{N}$ with $n< h(n)$ for all $n$, the following are equivalent:
\begin{itemize}
\item for every $\epsilon>0$ there is a $K$ so that for every $((X,d),(a_n)_{n\in\mathbb{N}})\in \mathcal{C}$, the sequence $(a_n)_{n\in\mathbb{N}}$ does not admit $K$ $\epsilon$-jumps of distance $h$,
\item whenever $\mathcal{U}$ is a nonprincipal ultrafilter on $\mathbb{N}$, and, for each $i$, $((X_i,d_i),(a^i_n)_{n\in\mathbb{N}})\in \mathcal{C}$, the sequence $(a_{\bar n})_{\bar n\in\mathbb{N}^*}$ converges in every cut in $(X_{\mathcal{U}},d_{\mathcal{U}})$ closed under $h$.
\end{itemize}
\end{theorem}
For instance, considering only additive cuts is considering the case where $h(n)=2n$---that is, the case where a sequence has only boundedly many functions spaced out by a multiplicative function.
\begin{proof}
  Suppose the former fails: there is some $\epsilon>0$ so that, for every $K$ there is an $((X,d),(a_n)_{n\in\mathbb{N}})\in \mathcal{C}$ so that $(a_n)_{n\in\mathbb{N}}$ admits $K$ $\epsilon$-jumps of distance $h$.  For each $K,r$, choose such an $((X^K,d^K),(a^K_n)_{n\in\mathbb{N}})$ and choose witnesses $n^K_1<n^K_2<\cdots<n^K_K$ so that, for each $k<K$, $n^K_{k+1}\geq h(n^K_k)$ and $d(a^K_{n^K_k},a^K_{n^K_{k+1}})\geq\epsilon$.

As above, take any nonprincipal ultrafilter $\mathcal{U}$ and consider the sequence $(a_{\bar n})_{\bar n\in\mathbb{N}^*}$.  For each $i\in\mathbb{N}$, take $\bar n_i=[\langle n^K_i\rangle_{K\in\mathbb{N}}]$ and let $I=\{\bar m\mid \exists i\ \bar m<\bar n_i\}$.  For any $\bar m\in I$, we have $\bar m<\bar n_i$ for some $i$, and therefore $\bar h(\bar m)\leq \bar h(\bar n_i)\leq \bar n_{i+1}\in I$, so $I$ is closed under $h$.  The remainder of the proof is as in the proof of Theorem \ref{thm:main}.

Conversely, suppose the former holds, and consider any nonprincipal ultrafilter $\mathcal{U}$, any sequence $((X_i,d_i),(a^i_n)_{n\in\mathbb{N}})$ of elements of $\mathcal{C}$.  Consider some cut $I$ closed under $h$ and some $\epsilon>0$.  Let $K$ witness the uniform bound for $\epsilon/4$-jumps.  Choose any $\bar n_1\in I$; then $\bar h(\bar n_1)\in I$, and if $\bar h(\bar n_1)$ does not witnesses convergence to within $\epsilon$, there is some $n_2\geq\bar h(n_1)$ with $d_{\mathcal{U}}(\bar n_1,\bar n_2)\geq\epsilon/2$.  We continue, choosing $\bar n_3\geq\bar h(\bar n_2)$ in $I$ and so on, and finish as in the proof of Theorem \ref{thm:main}.
\end{proof}

\section{Directions}

Avigad and Rute \cite{MR3345161} ask whether bounds for fluctuations exist for Walsh's generalization \cite{MR2912715} of Tao's nonconventional averages to polynomial actions of nilpotent groups.  Similar methods to those in the previous section might apply, particularly to Austin's proof \cite{MR3384494} of the result.

Although we only considered $L^2$ convergence, the same criterion gives bounds on upcrossings from pointwise convergence in all cuts.  Various papers \cite{MR3043589,MR2995648,2014arXiv1406.2608E,MR2753294,2014arXiv1406.5930H} have studied pointwise convergence of various nonconventional ergodic averages.  Any of these results might be adapted to nonstandard cuts, thereby obtaining upcrossing bounds.

It would be interesting to explicitly compare Kohlenbach and Safarik's result \cite{MR3111912} about existence of bounds on fluctuations to ours.  In particular, investigation of why their conditions imply convergence in all cuts might yield some insight on the relationship between provability in restricted systems and analogous results in nonstandard analysis.

\printbibliography
\end{document}